\documentclass[reqno, 11pt]{amsart}
\usepackage{url}
\usepackage[colorlinks]{hyperref}
\usepackage{times}

\newtheorem{theorem}{Theorem}[section]
\newtheorem{lemma}[theorem]{Lemma}

\newtheorem{corollary}[theorem]{Corollary}

\theoremstyle{definition}

\newtheorem{ex}[theorem]{Example}

\newtheorem{remark}[theorem]{Remark}
\newtheorem{question}{Question}

\numberwithin{equation}{section}

\usepackage{bbm}
\usepackage{url}
\usepackage{euscript}
\usepackage{pb-diagram}
\usepackage{lamsarrow}
\usepackage{pb-lams}
\usepackage{amsmath}
\usepackage{amsthm}

\usepackage{amsxtra}
\usepackage{amssymb}
\usepackage{pifont}
\usepackage{amsbsy}

\usepackage{graphicx}
\usepackage{epstopdf}

\newskip\aline \newskip\halfaline
\def\skipaline{\vskip\aline}

\def\qedbox{$\rlap{$\sqcap$}\sqcup$}
\def\qed{\nobreak\hfill\penalty250 \hbox{}\nobreak\hfill\qedbox\skipaline}

\newcommand\bC{{\mathbb C}}

\newcommand\bR{{\mathbb R}}

\newcommand\bZ{{\mathbb Z}}

\DeclareMathOperator{\sign}{sign}

\newcommand{\be}{{\boldsymbol{e}}}

\newcommand{\bh}{\boldsymbol{h}}

\newcommand{\bm}{\boldsymbol{m}}

\newcommand{\bv}{{\boldsymbol{v}}}

\newcommand{\bsC}{\boldsymbol{C}}
\newcommand{\bsD}{\boldsymbol{D}}

\newcommand{\bsK}{{\boldsymbol{K}}}
\newcommand{\bsL}{\boldsymbol{L}}

\newcommand{\bsO}{\boldsymbol{O}}
\newcommand{\bsP}{\boldsymbol{P}}

\newcommand{\bsS}{\boldsymbol{S}}
\newcommand{\bsT}{\boldsymbol{T}}

\newcommand{\bom}{{\boldsymbol{\omega}}}

\newcommand{\si}{{\sigma}}

\newcommand{\eps}{{\epsilon}}

\newcommand{\eF}{\EuScript{F}}
\newcommand{\eG}{\EuScript{G}}

\newcommand{\eH}{\EuScript H}
\newcommand{\eI}{{\EuScript{I}}}

\newcommand{\eO}{\EuScript{O}}

\newcommand{\eT}{\EuScript{T}}

\newcommand{\ra}{\rightarrow}

\newcommand{\Llra}{{\Longleftrightarrow}}

\def\inpr{\mathbin{\hbox to 6pt{\vrule height0.4pt width5pt depth0pt \kern-.4pt \vrule height6pt width0.4pt depth0pt\hss}}}

\newcommand{\pa}{\partial}

\begin{document}

\title{Combinatorial Morse flows are hard to find}

\date{Started  May 18, 2011. Completed  on January 27, 2012 .
Last modified on {\today}. }

\author{Liviu I. Nicolaescu}

\address{Department of Mathematics, University of Notre Dame, Notre Dame, IN 46556-4618.}
\email{nicolaescu.1@nd.edu}
\urladdr{\url{http://www.nd.edu/~lnicolae/}}

\begin{abstract}   We   investigate the probability of detecting  combinatorial Morse flows on a simplicial complex   via a random search.   We  prove that it is really small, in a quantifiable way.
\end{abstract}

\maketitle

\tableofcontents

\section{Introduction}
\setcounter{equation}{0}

 Let $X$ be  a compact  space equipped with a triangulation $\eF$. Here  $\eF$ stands for the collection of all the closed faces of the triagulation.   The collection $\eF$ is a poset      with  the order relation given by inclusion.   For  any function $f:\eF\ra \bR$,  and any  face $\si\in\eF$ we define
 \[
 A_{>\si}(f):=\bigl\{ \tau\in\eF;\;\;\dim\tau=\dim\si+1,\;\;f(\tau)\leq f(\si)\,\bigr\},
 \]
 \[
  A_{<\si}(f):=\bigl\{ \tau\in\eF;\;\;\dim\tau=\dim\si-1,\;\;f(\tau)\geq f(\si)\,\bigr\},
  \]
  \[
  A_\si(f):=A_{>\si}(f)\cup A_{<\si}(f).
  \]
  Following R. Forman \cite{For1},  we define a combinatorial   Morse function to be a function $f:\eF\to\bR$  such that
  \[
  |A_\si(f)|\leq 1,\;\;\forall  \si\in\eF.
  \]
  A face  $\si$ such that $|A_\si(f)|=0$ is called a \emph{critical face} of  the   combinatorial Morse function. Let us observe that the function
  \[
  \eF\ni \si\mapsto \dim\si
  \]
  is a combinatorial Morse function. All the faces are critical for this function.   
  
  Recall that the  \emph{Hasse diagram} of the     triangulation $\eF$ is the directed graph $\eH(\eF)$ whose vertex set is $\eF$, while the set of edges $E(\eF)$ is defined as follows: we have an edge  going from $\si\in\eF$ to $\tau\in\eF$ if and only if
  \[
  \dim\si-\dim\tau=1\;\;\mbox{and}\;\;\si\supset \tau.
  \]
     To any function $\bom:E(\eF)\to\{\pm 1\}$, and any face $\si\in \eF$ we associate the sets
  \[
  A_{>\si}(\bom):= \bigl\{  \tau\in\eF;\;\;\overrightarrow{\tau\si}\in E(\eF),\;\;\bom(\overrightarrow{\tau\si})=-1\,\bigr\},
  \]
  \[
  A_{<\si}(\bom):=  \bigl\{  \tau\in\eF;\;\;\overrightarrow{\si\tau}\in E(\eF),\;\;\bom(\overrightarrow{\si\tau})=-1\,\bigr\},
  \]
  \[
  A_\si(\bom)=A_{>\si}(\bom)\cup A_{<\si}(\bom).
  \]
 We    will  refer to a function $\bom:E(\eF,\bom)\to\{\pm 1\}$ as an \emph{orientation prescription} of $\eF$, and we will denote by $\eO_\eF$ the collection of all orientation prescriptions of $\eF$.
 
 Any orientation prescription $\bom$  defines    a new directed graph $\eH(\eF,\bom)$  whose   vertex set is $\eF$ and  its set of edges $E(\eF,\bom)$  is defined as follows.

\begin{itemize}    

\item The  undirected  graphs $\eH(\eF)_0$ and $\eH(\eF,\bom)_0$ have the same  sets of edges.

\item  If $\be$ is a \emph{directed} edge     of $\eH(\eF)$,   and $\bom(\be)=1$, then  $\be$ is an edge of $\eH(\eF,\bom)$. Otherwise,  switch the orientation of $\be$.

\end{itemize}

Any  combinatorial Morse function $f:\eF\ra \bR$ defines an orientation prescription $\bom_f: E(\eF)\ra \bR$ as follows.  If $\overrightarrow{\si\tau}$ is a directed   edge of  $\eH(\eF)$ then
\[
\bom_f\bigl(\, \overrightarrow{\si\tau}\,\bigr)=- 1 \Llra  \tau\in A_{<\si}(f)
\]
\[
\Llra  \dim\si-\dim\tau=1,\;\;\tau\subset\si,\;\;f(\tau)\geq f(\si).
\]
Observe that the Morse condition implies that   the directed graph $\eH(\eF,\bom_f)$ has no (directed) cycles.  Moreover
\[
A_{>\si}(\bom_f)= A_{>\si}(f),\;\;A_{<\si}(\bom_f)=A_{<\si}(f).
\]
We define a \emph{combinatorial flow} on $\eF$ to be an orientation prescription  $\bom: E(\eF)\to\bR$ such that
\[
|A_\si(\bom)|=1,\;\;\forall\si\in\eF.
\]
If  $\bom$ defines a combinatorial flow, then the set of directed edges $\be \in E(F)$  such that $\bom(\be)=-1$ define a matching (in the sense of \cite[Def. 11.1]{Koz}) of the poset of faces   $\eF$. 

Observe that the orientation prescription  determined by a  combinatorial Morse  function is  a combinatorial flow. We will refer to such  flows  as \emph{combinatorial Morse flows}.  Conversely, \cite[Thm. 11.2]{Koz},   a combinatorial flow is Morse  if and only if it is acyclic, i.e.,    the  directed graph $\eH(\eF,\bom)$ is acyclic.\footnote{In the paper \cite{KB} that precedes  R. Forman's work,  K. Brown introduced  the concept  of collapsing scheme  which   identical to the above concept of acyclic matching.}  In topological applications   the combinatorial flow   determined by a    combinatorial Morse function   plays the key role. In fact, once  we have an acyclic   combinatorial flow  one can very easily produce a   Morse function generating it.  A natural question then  arises.

\smallskip

\emph{How  can one produce   acyclic combinatorial flows?}

\smallskip

The  present paper grew out of our attempts to  answer this question.  Here is   a simple strategy. Suppose that by some means we have detected an orientation prescription  $\bom$ that generates a combinatorial flow. We denote by $\sign\bom$ the number of edges $\be\in E(\eF)$ such that $\bom(\be)=-1$. We will  deform $\bom$ to an acyclic flow  using the following procedure.

\medskip

\noindent {\bf Step 1.}  If $\eH(\eF,\bom)$ is acyclic, then  STOP.

\smallskip

\noindent  {\bf  Step 2.}    If $\eH(\eF,\bom)$   contain cycles, choose one.     Then     at least one of the edges along this  cycle  belongs to  the set $\{\bom=-1\}$.   Chose one such edge $\be$ and define a new orientation prescription  $\bom'$ which is equal to $\bom$ on any edge other that $\be$,  whereas $\bom'(\be)=-\bom(\be)=1$. Note that $\sign \bom'=\sign \bom-1$.   GOTO  {\bf Step 1.}

\medskip

The above procedure reduces the problem to  producing combinatorial flows.        We have attempted a probabilistic  approach.    Switch  randomly and independently  the orientations of  the edges  in $E(F)$.    How likely is it that the resulting   orientation prescription  defines a combinatorial flow?

 More precisely, we equip  the set of orientation prescriptions  $\eO_\eF$ with the uniform  probability  measure, and  we denote by $\bsP(\eF)$ the probability  that a random orientation  prescription is a   combinatorial flow.      We are interested in estimating this probability when $\eF$ is large.  
 
  Note that if $\eF'$ denotes a subcomplex of $\eF$, then $\bsP(\eF)\leq \bsP(\eF')$.  In particular,  if $\eF_1$  denotes  the $1$-skeleton of the  triangulation, then $\bsP(\eF)\leq \bsP(\eF_1)$. For this reason   we will concentrate excusively on $1$-dimensional complexes, i.e., graphs.

  Consider a  graph $\Gamma$ with vertex set $V(\Gamma)$ and  edge\footnote{We do not allow loops or multiple edges between  a pair of vertices.} set $E(\Gamma)$.   If $E(\Gamma)=\emptyset$, i.e., $\Gamma$ consists of isolated points, then trivially $\bsP(\Gamma)=1$.     
  
  Suppose  that $E(\Gamma)\neq\emptyset$.   If we regard  $\Gamma$ as a $1$-dimensional simplicial complex, then its   barycentric subdivision is the graph  $\Gamma'$ obtained  by  marking  the midpoints  of the edges of $\Gamma$.    The  vertices of $\eH(\Gamma)$, the Hasse  diagram of $\Gamma$, consist of the vertices $\bv$ of $\Gamma$ together with the midpoints $b_\be$  of the edges $\be$ of $\Gamma$.  To  each edge $\be$ of $\Gamma$  one associates a pair of  directed edges  of the Hasse diagram, running from the midpoint $b_\be$ of that edge towards the   endpoints of that edge.
  
  Define the incidence relation $\eI_\Gamma\subset V(\Gamma)\times E(\Gamma)$ where $(\bv,\be)\in \eI_\Gamma$ if and only if  the vertex $\bv$ is an endpoint of the edge $\be$.  Note that $\eI_\Gamma$ can be identified with the set of edges of the barycentric subdivision $\Gamma'$. An orientation prescription is then a function $\bom:\eI_\Gamma\ra \{\pm 1\}$.    The edge $(\bv, b_\be)$   of $\Gamma' $ is given the orientation   $b_\be\to \bv$ in the digraph $\eH(\Gamma,\bom)$ if and only if $\bom(\bv,\be)=1$. 
 
 We denote by $\eO_\Gamma$ the set of orientation prescriptions on $\Gamma$ and by $\Phi_\Gamma\subset \eO_\Gamma$ the set of combinatorial flows.    Thus, an orientation prescription $\bom$ defines a combinatorial  flow if  the digraph $\eH(\Gamma,\bom)$ has the property   at each $\bv\in V(\Gamma)$ there exists at most one outgoing edge, and at each barycenter $b_\be$  there exists at most one incoming edge. 
 
 In  Figure \ref{fig: 1} we have described  orientation prescriptions on the simplical complex  defined by the boundary of a square.   The orientation prescription in the left-hand side  defines an acyclic combinatorial flow, and the numbers  assigned to the various vertices describe a combinatorial Morse function  defining this flow. The orientation prescription in the right- hand side  does not determine  a combinatorial flow.
 
 \begin{figure}[h]
\centering{\includegraphics[height=1.5in,width=2.5in]{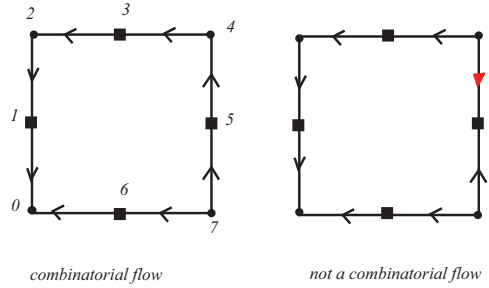}}
\caption{\sl  The vertices  of $\Gamma$ are marked with $\bullet$'s and the barycenters  of the edges are marked with {\tiny{$\blacksquare$}}'s.}
\label{fig: 1}
\end{figure}
 
  We denote  by $\bsP(\Gamma)$ the probability that an orientation prescription is a  combinatorial flow, i.e., 
 \[
 \bsP(\Gamma)=\frac{|\Phi_\Gamma|}{|\eO_\Gamma|}=\frac{|\Phi_\Gamma|}{4^{N_\Gamma}},\;\; N_\Gamma= |E(\Gamma)|.
 \]
 The above definition implies trivially that
 \begin{equation} 
 \bsP(\Gamma)\geq \frac{1}{4^{N_\Gamma}}.
 \label{eq: lower}
 \end{equation}
Consider the graph $\bsL_1$ consisting of two vertices $v_0,v_1$ connected by an edge.    It is easy to see  that (see Figure \ref{fig: 2})
  \[
  \bsP(\bsL_1)= \frac{3}{4}.
  \]

\begin{figure}[h]
\centering{\includegraphics[height=0.8in,width=3in]{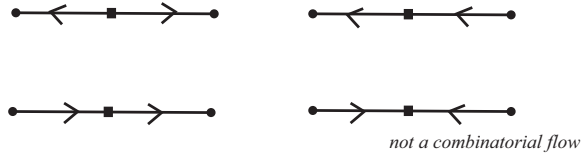}}
\caption{\sl  There are  four  orientation prescriptions on $\bsL_1$ and only one of them is not  a combinatorial flow}
\label{fig: 2}
\end{figure}

  If $\bom$ is an orientation prescription on  a graph $\Gamma$, then it defines a combinatorial flow only if its restriction to each of the edges (viewed as copies of $\bsL_1$) are combinatorial flows.   We deduce 
  \begin{equation}
  \bsP(\Gamma)\leq \left(\frac{3}{4}\right)^{N_\Gamma}
  \label{eq: upper}
  \end{equation}
  Note that  the above  upper bound is  optimal:   we have equality when $\Gamma$ consists of disjoint edges. Already this  shows  that  the above probabilistic approach  has  very small  chances   of success.  However, we wish to   say  something more.
  
   Motivated by the estimates (\ref{eq: lower}) and (\ref{eq: upper}) we introduce a new invariant
 \[
 \bh(\Gamma):=\begin{cases}\frac{\log \bsP(\Gamma)}{N_\Gamma}, & E(\Gamma)\neq \emptyset \\
 0, & E(\Gamma)=\emptyset.
 \end{cases}
 \] 
 The   inequalities  (\ref{eq: lower}) and (\ref{eq: upper}) can be rewritten as
 \begin{equation}
E(\Gamma)\neq\emptyset \Rightarrow \log\left(\frac{1}{4}\right)\leq \bh(\Gamma)\leq \log\left(\frac{3}{4}\right).
 \label{eq: lo-up}
 \end{equation}
 In this  paper we  investigate  the invariant   $\bh(\Gamma)$ for various  classes of graphs and study its behavior  as $\Gamma$ becomes very large. In particular, we prove that  the inequalities (\ref{eq: lo-up}) are optimal.

The above lower bound  is  also an asymptotically   optimal bound. More precisely the arguments in this paper show  that 
\[
\log\left(\frac{1}{4}\right)=\liminf_{\substack{N_\Gamma\to\infty,\\  b_0(\Gamma)=1}}  \bh(\Gamma),
\]
where $b_0(\Gamma)$ denotes the number of connected components of $\Gamma$. The same cannot be said about the upper bound.   in is not hard to see  that
\[
\limsup_{\substack{N_\Gamma\to\infty,\\  b_0(\Gamma)=1}}  \bh(\Gamma) <  \log\left(\frac{3}{4}\right).
\]
Moreover,  the results in  Section \ref{s: special} show that
\[
\limsup_{\substack{N_\Gamma\to\infty,\\  b_0(\Gamma)=1}}  \bh(\Gamma)  \geq \log\left( \frac{3+\sqrt{5}}{8}\right).
\]
We are inclined to believe  that  in fact we have equality  above.  

The paper is structured as follows. In Section \ref{s: 2} we describe several  general techniques for computing  $\bsP(\Gamma)$. In Section \ref{s: special} we use these general techniques to compute   $\bsP(\Gamma)$ for  several classes of graphs $\Gamma$. In  Section \ref{s: 4}  we  describe several general  properties  of $\bh(\Gamma)$ and formulate several   problems that we believe are interesting.

 \section{General facts concerning combinatorial  flows on  graphs}
 \label{s: 2}
 \setcounter{equation}{0}

 Consider a  graph $\Gamma$ with vertex set $V(\Gamma)$ and  edge set $E(\Gamma)\neq\emptyset$.    To   $\bom\in\eO_\Gamma$  we associate an \emph{anomaly} function $A_\bom: V(\Gamma)\ra \bZ_{\geq 0}$,  where  for any $\bv\in V(\Gamma)$ we denote by $A_\bom(\bv)$ the number of    edges of the digraph $\eH(\Gamma,\bom)$ that exit the vertex $\bv$.  For  any  subset $S\subset V(\Gamma)$  and any function $f:S \ra \bZ_{\geq 0}$ we denote by $\bsP_S(\Gamma|\,f\,)$ the conditional probability that the orientation prescription $\bom\in\eO_\Gamma$ is a combinatorial flow  given that $A_\bom|_S=f$. Note that $\bsP_S(\Gamma| f)=0$ if $\max f>1$ and 
\begin{equation}
\bsP(\Gamma)=\sum_{f:V(\Gamma)\to\{0,1\}} \bsP_{V(\Gamma)}\bigl(\,\Gamma\;\bigl|\; f\,\bigr).
\label{eq: tot}
\end{equation}
 The above  conditional probabilities satisfy  two very simple  rules, the \emph{product rule} and the \emph{quotient rule}.

The product rule explains what happens with the various probabilities when we take the disjoint union of  two graphs. More precisely, suppose we are given disjoint graphs $\Gamma_i$, subsets  $S_i\subset V(\Gamma_i)$ and functions $f_i:S_i\to\bZ_{\geq 0}$, $i=1,2$.  Then
\begin{equation}
\bsP_{S_1\sqcup S_2}\bigl(\Gamma_1\sqcup \Gamma_2\,\bigl|\;f_1\sqcup f_2\,\bigl)=\bsP_{S_1}\bigl(\,\Gamma_1\;\bigl|\;f_1\,\bigr)\cdot \bsP_{S_2}\bigl(\,\Gamma_2\;\bigl|\;f_2\,\bigr), 
\label{eq: prod}
\end{equation}
The product rule explains what happens with the various probabilities when  we   identify   several vertices in a graph, thus obtaining a new graph with fewer vertices but the same number of edges.

Suppose that we are given a graph  $\Gamma$ and an equivalence relation   "$\sim$" on   $V(\Gamma)$. Denote by $\bar{\Gamma}$ the graph obtained from $\Gamma$ by identifying vertices  via the equivalence relation $\sim$.  Denote by $\pi$ the natural  projection
\[
 \pi: V(\Gamma)\to V(\Gamma)/\sim=V(\bar{\Gamma}).
 \]
  Fix a     subset $\bar{S}\subset V(\bar{\Gamma})$ and a function $\bar{f}:\bar{S}\to\bZ_{\geq 0}$. We denote by $S$ the preimage $S:=\pi^{-1}(S)$. To any function $g:S\to \bR$  we associate a function 
  \[
  \pi_*(g):\bar{S}\to\bR,
  \]
  obtained by integrating $g$ along the fibers of $\pi$, i.e.
  \[
  \pi_*(g)(\bar{s}):=\sum_{s\in\pi^{-1}(\bar{s})} g(s),\;\;\forall \bar{s}\in\bar{S}.
  \]
  The quotient rule the states
  \begin{equation}
  \bsP_{\bar{S}}\bigl(\,\bar{\Gamma}\;\bigl|\; \bar{f}\;\bigl)=\sum_{\pi_*(f)=f}\bsP_S\bigl(\,\Gamma\;\bigl|\; f\,\bigr),\;\;\forall \bar{f}:\bar{S}\ra \{0,1\}.
  \label{eq: quot}
  \end{equation}
  In particular
  \begin{equation}
  \bsP(\bar{\Gamma})=\sum_{\pi_*(f)\leq 1}\bsP_{V(\Gamma)}\bigl(\,\Gamma\;\bigl|\; f\,\bigr). 
  \label{eq: quot1}
  \end{equation}
  
  \begin{ex} Consider the graph $\bsL_1$ consisting of two vertices $v_0,v_1$ connected by an edge.   A function $\eps: V(\bsL_1)\to\{0,1\}$ is determined by two numbers $\eps_i=\eps(v_i)$.
We set
  \[
  p_1(\eps_0,\eps_1):=\bsP_{V(\bsL_1)}(\bsL_1\,|\; \eps).
  \]
  An inspection of  Figure \ref{fig: 2} shows that
  \[
p_1(0,1)=p_1(1,0)=p_1(0,0)=\frac{1}{4}, \;\;p_1(1,1)=0.
\]
 \qed
\label{ex:edge}
  \end{ex}

  Note that every graph with $n$ edges  is a  quotient of the graph  consisting of   $n$  disjoint copies of $\bsL_1$, we can  use  (\ref{eq: tot}), (\ref{eq: prod}) and (\ref{eq: quot}) a produce a formula formula for  $\bsP(\Gamma)$.
  
  Given a graph $\Gamma$ we introduce  formal variables $\vec{z}:=(z_v)_{v\in V(\Gamma)}$. To an edge  $\be$ of   $\Gamma$ with endpoints  $v_0,v_1$  we associate the polynomial
  \[
Q_\be(\vec{z}):=  \sum_{\eps_0,\eps_1\leq 1}  p_1(\eps_0,\eps_1) z_{v_0}^{\eps_0}z_{v_1}^{\eps_1}=\frac{1}{4}(1+z_0+z_1)=\frac{1}{4}\bigl(\, (1+z_0)(1+z_1)-z_0z_1\,\bigr).
  \]
 We define
 \[
 Q_\Gamma(\vec{z}):=\prod_{\be\in E(\Gamma)}Q_\be.
 \]
 Then the quotient rule (\ref{eq: quot1})  implies 
 \begin{equation}
 \bsP(\Gamma)=\sum_{S\subset V(\Gamma)} \pa_SQ_\Gamma|_{\vec{z}=0},
 \label{eq: prob1}
 \end{equation}
 where for any subset $S=\{v_1,\dotsc,v_k\}\subset V(\Gamma)$ we define
 \[
\pa_S:=\frac{\pa^k}{\pa z_{v_1}\dotsc \pa z_{v_k}}.
 \]
 Observe that the term $ \pa_S Q_\Gamma|_{\vec{z}=0}$  involves only the subgraph $\widetilde{\Gamma}_S$ of $\Gamma$ formed by the edges incident to the vertices in $S$.
 
 It is convenient to regard $Q_\Gamma$ as a (polynomial) function on the vector space $\bC^{V(\Gamma)}$ with coordinates $(z_v)_{v\in V(\Gamma)}$.   If $\sim$ is an equivalence relation  on $V(\Gamma)$ and $\bar{\Gamma}$ denotes the graph $\Gamma/\sim$, then we can identify $\bC^{V(\bar{\Gamma})}$ with the subspace of $\bC^{V(\Gamma)}$ given by the linear equations
 \[
 z_{u}=z_v \Llra u\sim v.
 \]
  Moreover
  \[
  Q_{\bar{\Gamma}}=Q_{\Gamma}|_{\bC^{V(\bar{\Gamma})}}.
  \]
For any multi-index $\alpha\in \bZ_{\geq 0}^{V(\Gamma)}$ we set
\[
\vec{z}^\alpha= \prod_v z_v^{\alpha_v}.
\]
For  any polynomial  
\[
P=\sum_\alpha p_\alpha \vec{z}^\alpha\in \bC[ (z_v)_{v\in V(\Gamma)}]
\]
 we define its \emph{truncation}
 \[
 \bsT[P]=\sum_{\alpha_v\leq 1} p_\alpha \vec{z}^\alpha.
 \]
 Any subset  $ S\subset V(\Gamma)$  defines a   multi-index $\alpha=\alpha_S\in \bZ_{\geq 0}^{V(\Gamma)}$, $\alpha_v=1$ if $v\in S$, $\alpha_v=0$ if $v\not\in S$.  We write
\[
\vec{z}^S:=\vec{z}^{\alpha_S},\;\; p_S:=p_{\alpha_S},
\]
so that the truncated polynomial has the form
\[
\bsT [P]= \sum_{S\subset V(\Gamma)}p_S \vec{z}^S.
\]
The equality (\ref{eq: prob1}) can be rewritten as
\begin{equation}
\bsP(\Gamma)=\bsT [Q_\Gamma](1) .
\label{eq: prob2}
\end{equation}

\section{Combinatorial  flows on   various classes of graphs}
\label{s: special}
\setcounter{equation}{0}
In the sequel we will denote by $I_n$ the set $\{1,\dotsc, n\}$.
\begin{theorem} Denote by $\bsS_n$ the star shaped graph consisting of $n+1$ vertices $v_0,v_1,\dotsc, v_n$ and $n$ edges $[v_0,v_1],\dotsc, [v_0,v_n]$. Then
\begin{equation}
 \bsT [Q_{\bsS_n}]=\frac{1}{4^n}\sum_{S\subset I_n} \vec{z}^S +\frac{1}{4^n}\sum_{S\subset I_n} (n-|S|) z_0\vec{z}^S,
 \label{eq: star1}
 \end{equation}
 \begin{equation}
 \bsP(\bsS_n)=\frac{n+2}{2^{n+1}},
 \label{eq: star2}
 \end{equation}
 and 
 \begin{equation}
 \bh(\bsS_n)\sim \log\left(\frac{1}{2}\right) \;\;\mbox{as $n\to\infty$}.
 \label{eq: star-asy}
 \end{equation}
 \label{th: star}
 \end{theorem}
 \begin{proof} We have
 \[
 4^n Q_{\bsS_n}=\sum_{S\subset I_n} \prod_{i\in S} (z_0+z_i)
 \]
 so that
 \[
 4^n \bsT[Q_{\bsS_n}]=\sum_{S\subset I_n} \bsT\left[\prod_{i\in S} (z_0+z_i)\right]
 \]
 \[
 \sum_{S\subset I_n} \left( \prod_{i\in S} z_i +z_0\sum_{j\in S} \prod_{i\in S\setminus j} z_i\right)=  \sum_{S\subset I_n} \vec{z}^S +z_0\sum_{S\subset I_n} (n-|S|) \vec{z}^S.
 \]
 Hence 
 \[
 \bsP(\bsS_n)=\frac{1}{4^n}\sum_{k=0}^n(n-k+1)\binom{n}{k} =\frac{1}{2^n}+\frac{1}{4^n}\sum_{k=0}^n (n-k) \binom{n}{n-k}
 \]
 \[
 =\frac{1}{2^n}+\frac{1}{4^n}\underbrace{\sum_{j=1}^n j\binom{n}{j}}_{=n2^{n-1}}=\frac{n+2}{2^{n+1}}.
 \]
 The estimate (\ref{eq: star-asy})  is now obvious.
 \end{proof}

\begin{theorem} Denote by $\bsL_n$ the graph with $n+1$-vertices $v_0,v_1\dotsc, v_n$ and $n$ edges
 \[
 [v_0,v_1],[v_1,v_2],\dotsc,  [v_{n-1},v_n].
 \]
 We set $p_n=p(\bsL_n)$  and 
 \[
p_n:=p(\bsL_n),\;\;p(z):= \sum_{n\geq 1} p_nz^n.
\]
Then
\begin{equation}
\sum_{n\geq 1} p_nz^n= \frac{12z-z^2}{(z^2-12z+16)}=\frac{16}{(z^2-12z+16)}-1.
\label{eq: gen-fun0}
\end{equation}
In particular
 \begin{equation}
\bh(\bsL_n)\sim \log r\;\;\mbox{as $n\to\infty$},
 \label{eq: asy1}
 \end{equation}
 where 
 \begin{equation}
 r=\frac{3+\sqrt{5}}{8}\approx 0.654 <\frac{3}{4}.
 \label{eq: r}
 \end{equation}
 \label{th: key}
 \end{theorem}

\begin{proof}  For $\eps,\eps'\in\{0,1\}$ we set
\[
p_n(\eps)=\bsP\bigl(\bom\in\Phi_{\Gamma_n}\;\bigl|\; A_\bom(v_0)=\eps\,\bigr)= \bsP\bigl(\bom\in\Phi_{\Gamma_n}\;\bigl|\; A_\bom(v_n)=\eps\,\bigr),
\]
\[
p_n(\eps,\eps')= \bsP\bigl(\bom\in\Phi_{\Gamma_n}\;\bigl|\; A_\bom(v_0)=\eps, A_\bom(v_n)=\eps'\,\bigr),
\]
\[
\vec{p_n}:=  \left[
\begin{array}{c}
p_n(0)\\
p_n(1)
\end{array}
\right],\;\;\vec{p_n}(\eps):=  \left[
\begin{array}{c}
p_n(0,\eps)\\
p_n(1.\eps)
\end{array}
\right].
\]
Hence $p_n=p_n(0)+p_n(1)$. Note that 
\[
p_1(0)=p_1(0,0)+p_1(0,1)=\frac{1}{2},\;\;p_1(1)=p_1(1,0)+p_1(1,1)=\frac{1}{4}.
\]
The equality  (\ref{eq: prod}) implies
\[
p_n(0)=\sum_{\eps+\eps'\leq 1}p_1(0,\eps)p_{n-1}(\eps')=\frac{1}{4}\left(\,2 p_{n-1}(0)+ p_{n-1}(1)\,\right),
\]
\[
p_{n}(1)= \sum_{\eps+\eps'\leq 1} p_1(1,\eps)p_{n-1}(\eps')=\frac{1}{4}\left(\, p_{n-1}(0)+ p_{n-1}(1)\,\right).
\]
 We can rewrite the above equalities in the compact form
 \[
 \vec{p}_n= A\vec{p}_{n-1},\;\;A:=\frac{1}{4}\left[
 \begin{array}{cc}
 2 & 1\\
 1 & 1
 \end{array}
 \right].
 \]
 We deduce
 \begin{equation}
 \vec{p}_n=A^{n-1}\vec{p}_1.
 \label{eq: recurr}
 \end{equation}
 We conclude similarly that
 \begin{equation}
 \vec{p}_n(\eps)=A^{n-1}\vec{p}_1(\eps).
 \label{eq: recurr*}
 \end{equation}
 The characteristic polynomial of  $A$ is
 \[
 \lambda^2-\frac{3}{4}\lambda+\frac{1}{16} =0,
 \]
 and its eigenvalues are 
 \[
 \lambda_{\pm}= \frac{3\pm \sqrt{5}}{8}.
 \]
 Each of the sequences   $p_n(\eps)$    is a solution of the second order linear recurrence relation
 \begin{equation}
 x_{n+2}-\frac{3}{4}x_{n+1}+\frac{1}{16}x_n=0
 \label{eq: recurr1}
 \end{equation}
 We deduce that $p_n$ also satisfies the above linear recurrence relation so   that
 \[
 p(z)= \frac{C'z^2+B'z}{\frac{1}{16}z^2-\frac{3}{4}z+1}= \frac{Cz^2+Bz}{z^2-12z+16},
 \]
 where $C=16C'$, $B=16B'$ are real constants. Note that
\[
\vec{p}_2=A\vec{p}_1=\frac{1}{16}\left[
 \begin{array}{cc}
 2 & 1\\
 1 & 1
 \end{array}
 \right]\cdot\left[\begin{array}{c}
 2\\
 1
 \end{array}
 \right]=\left[
 \begin{array}{c}
 \frac{5}{16}\\
 \\
 \frac{3}{16}
 \end{array}
 \right].
 \]
 Hence $p_2=\frac{8}{16}$.    Now observe that
 \[
 \frac{1}{z^2-12z+16}={\frac {1}{16}}+{\frac {3}{64}}z+O \left( {z}^
{2} \right).
\]
Hence
\[
p(z)=\frac{B}{16}z+ \left(\frac{3B}{64}+\frac{C}{16}\right)z^2 +O \left( {z}^
{3} \right).
\]
 We deduce that $B=12$, $C=-1$.  The estimate (\ref{eq: asy1})  follows from the above discussion.
\end{proof}

\begin{remark} (a) Using   MAPLE we can   easily determine the first few values of $p_n$. We have
\[
p(z)={\frac {3}{4}}z+{\frac {1}{2}}{z}^{2}+{\frac {21}{64}}{z}^{3}+{\frac 
{55}{256}}{z}^{4}+{\frac {9}{64}}{z}^{5}+{\frac {377}{4096}}{z}^{6}+{
\frac {987}{16384}}{z}^{7}+{\frac {323}{8192}}{z}^{8}+O \left( {z}^{9}
 \right) .
\]
Ultimately, the recurrence (\ref{eq: recurr1}) is the fastest way to compute $p_n$ for any $n$.

(b)  Note that $\bsS_2=\bsL_2$. In this case the equality (\ref{eq: star2}) is in perfect  agreement with  the equality $\bsP(\bsL_2)=\frac{1}{2}$.\qed
\end{remark}

\begin{ex}[Octopi and dandelions] (a) We  define an \emph{octopus} of type $(n_1,\dotsc,n_k)$, $k\geq 3$,   to be the graph $\bsO(n_1,\dotsc, n_k)$ obtained by gluing the linear graphs $\bsL_{n_1},\dotsc, \bsL_{n_k}$  at a common endpoint.  The  quotient rule  implies that
 \begin{equation}
 \begin{split}
 \bsP\bigl(\, \bsO(n_1,\dotsc, n_k)\,\bigr)=\prod_{j=1}^n p_{n_j}(0)  +\sum_{j=1}^n  p_{n_1}(0)\cdots p_{n_{j-1}}(0) p_{n_j}(1)p_{n_{j+1}}(0)\cdots p_{n_k}(0)\\
 = \prod_{j=1}^n p_{n_j}(0) \left(\, 1+\sum_{j=1}^k\frac{p_{n_j}(1)}{p_{n_j}(0)}\,\right).
 \end{split}
 \label{eq: oct}
 \end{equation}
 We write 
 \[
 \bsO_{k\times n}:=\bsO_(\underbrace{n_1,\dotsc, n_k}_{k}). 
 \]
 We deduce
 \[
 \bsP\bigl(\, \bsO_{k\times n}\,\bigr)=p_n(0)^k\left(1+ k\frac{p_n(1)}{p_n(0)}\right),\;\;\bh\bigl(\bsO_{k\times n}\,\bigr)=\frac{\log p_n(0)}{n} +\frac{\log\left(1+ k\frac{p_n(1)}{p_n(0)}\right)}{nk}.
 \]
(b) The \emph{dandelion}  of type $(n,m)$ is the graph 
\[
\bsD_{n,m}=\bsO(\, n, \;\underbrace{1,\dotsc,1}_m \,) .
\]
Using (\ref{eq: oct}) we deduce
\[
\bsP(\bsD_{n,m})= p_n(0)\left(\frac{1}{2}\right)^m\left(1+\frac{p_n(1)}{p_n(0)} + \frac{m}{2}\right).
\]
\qed
\end{ex}

 \begin{theorem} For $n\geq 3$ we denote by $\bsC_n$ the cyclic graph with $n$-vertices, i.e., the graph with vertices  $v_1,\dotsc, v_n$ and edges
 \[
 [v_1,v_2],\dotsc, [v_{n-1}, v_n], [v_n,v_1]
 \]
 Then   the sequence $\bsP(\bsC_n)$ satisfies the  linear recurrence relation (\ref{eq: recurr1}) with initial conditions
 \begin{equation}
 \bsP(\bsC_3)=\frac{9}{32},\;\; \bsP(\bsC_3)=\frac{47}{256}.
 \label{eq: init1}
 \end{equation}
 In particular
 \begin{equation}
 \bh(\bsC_n)\sim \log r\;\;\mbox{as $n\to\infty$},
 \label{eq: cycle}
 \end{equation}
 where $r$ is described by (\ref{eq: r}).
 \label{th: cyclic}
 \end{theorem}
 
 \begin{proof} The graph $\bsC_n$ is obtained from $\bsL_n$ by identifying the endpoints $v_0,v_n$ of $\bsL_n$. Using  (\ref{eq: quot}) and the notations in the proof of Theorem \ref{th: key} we deduce  that
 \[
 \bsP(\bsC_n)=\sum_{\eps+\eps'\leq 1}p_n(\eps,\eps')  =p_n(0,0)+ p_n(0,1)+p_n(1,0)= p_n(0,0)+2p_n(0,1).
 \]
 This shows that $\bsP(\bsC_n)$ satisfies the recurrence (\ref{eq: recurr1}) since both $p_n(0,0)$ and $p_n(0,1)$ do. Using (\ref{eq: recurr*}) we deduce
 \[
 \vec{p}_3(0)=\frac{1}{4^2}A^2\vec{p}_1(0)= \frac{1}{4^3}\left[
 \begin{array}{cc}
 5 & 3\\
 3 & 2
 \end{array}
 \right]\cdot \left[
 \begin{array}{c}
 1\\
 1
 \end{array}
 \right]=\left[
 \begin{array}{c}
 \frac{1}{8}\\
 \\
 \frac{5}{4^3}
 \end{array}
 \right]
 \]
 \[
 \vec{p}_3(0)=\frac{1}{4^4}\left[
 \begin{array}{cc}
 13 & 8\\
 8 & 5
 \end{array}
 \right]\cdot \left[
 \begin{array}{c}
 1\\
 1
 \end{array}
 \right]=\left[
 \begin{array}{c}
 \frac{21}{4^4}\\
 \\
 \frac{13}{4^4}
 \end{array}
 \right].
 \]
 This shows that
 \[
 \bsP(\bsC_3)= \frac{1}{4}+\frac{10}{4^3}=\frac{9}{32},\;\;\bsP(\bsC_4)=\frac{21}{4^4}+\frac{26}{4^4}=\frac{47}{256}=0.18359375.
 \]
 \end{proof}

 \begin{theorem} Denote by $\bsK_n$ the complete graph with $n$ vertices.  Then
 \begin{equation}
\frac{1}{4^{\frac{n(n-1)}{2}}} \left(1+\frac{n}{2}\right)^n\leq  \bsP(\bsK^n)\leq  \frac{1}{4^{\frac{n(n-1)}{2}}}(n+1)^n.
 \label{eq: comp1}
 \end{equation}
 In particular
 \begin{equation}
 \bh(\bsK_n)\sim \log\left(\frac{1}{4}\right)\;\;\mbox{as $n\to\infty$}.
 \label{eq: comp-asy}
 \end{equation}

 \label{th: comp}
 \end{theorem}
 
 \begin{proof} We have
 \[
 Q_{\bsK_2}= \frac{1}{4}\bsT[Q_{\bsK_2}]=  \frac{1}{4}( 1+z_1+z_2),
 \]
 \begin{equation}
 \begin{split}
 \bsT[Q_{\bsK_3}]= \frac{1}{4^3}\bsT\left[\,(1+z_1+z_2)\cdot \bsT\bigl[ (1+z_1+z_3)(1+z_2+z_3)\,\bigr]\,\right]\\
 =\frac{1}{4^3}\bsT\left[\,(1+z_1+z_2)\cdot (1+z_1+z_2+2z_3+ z_1z_2+z_2z_3+z_3z_1)\right]\\
 =\frac{1}{4^3}\bigl(\, 1+ 2(z_1+z_2+z_3)+ 3(z_1z_2+z_2z_3+z_3z_1)+ 2z_1z_2z_3\,\bigr).
 \end{split}
 \label{eq: tcom6}
 \end{equation}
 In general, we write
 \begin{equation}
 \bsT[Q_{\bsK_n}]=\frac{1}{4^{\frac{n(n-1)}{2}}}\sum_{S\subset I_n} c_n(S) \vec{z}^S,
 \label{eq: tcom1}
 \end{equation}
 where we recall that
 \[
 N_{\bsK_n}=\binom{n}{2}=\frac{n(n-1)}{2}.
 \]
 Observe that
 \[
 c_n(S)=c_n(S') \;\;\mbox{if}\;\;|S|=|S'|.
 \]
 We denote by $c_n(k)$ the common value of the numbers $c_n(S)$, $|S|=k$. We can rewrite (\ref{eq: tcom1}) as
 \begin{equation}
 \bsT[Q_{\bsK_n}]=\frac{1}{4^{\frac{n(n-1)}{2}}}\sum_{k=0}^n c_n(k)\left(\sum_{S\subset I_n,\;|S|=k}  \vec{z}^S\right).
 \label{eq: tcom2}
 \end{equation}
 In particular, we deduce that
 \begin{equation}
 \bsP(\bsK_n)=\frac{1}{4^{\frac{n(n-1)}{2}}}\sum_{k=0}^n \binom{n}{k}c_n(k).
 \label{eq: tcom7}
 \end{equation}
 Now think of the graph $\bsK_n$ as obtained  from the graph $\bsK_n$ as obtained from $\bsK_n$ by adding a new vertex $v_0$ and $n$-new edges $[v_0,v_k]$, $k=1,\dotsc, n$.  In other words $\bsK_{n+1}$ is a quotient of the graph $\bsK_n\sqcup\bsS_n$.  Using the product and quotient rules we deduce that 
 \begin{equation}
 \bsT\bigl[\,Q_{\bsK_{n+1}}(z_0,\dotsc, z_n)\,\bigr]=\bsT\bigl[\; \bsT[Q_{\bsK_n}](z_1,\dotsc, z_n)\cdot \bsT[Q_{\bsS_n}](z_0,z_1,\dotsc, z_n)\;\bigr].
 \label{eq: tcom3b}
 \end{equation}
For $S\subset I_n$, $|S|=k$, we deduce from (\ref{eq: star1}), (\ref{eq: tcom1})  and (\ref{eq: tcom3b}) that
\begin{equation}
c_{n+1}(k)=c_{n+1}(S)=\sum_{S'\sqcup S''=S} c_n(S')=\sum_{j=0}^k \binom{k}{j}c_n(j).
\label{eq: tcom3}
\end{equation}
If $I_n^*= \{0\}\cup I_n$, then   (\ref{eq: star1}), (\ref{eq: tcom1})  and (\ref{eq: tcom3b}) imply that
\begin{equation}
c_{n+1}(n+1)= c_n(I_n^*)=\sum_{S\sqcup S'=I_n} (n-|S|)c_n(S')=\sum_{k=0}^n k\binom{n}{k}c_n(k).
 \label{eq: tcom4}
 \end{equation}

\begin{lemma} For any $n\geq 3 $ and any $0\leq k\leq n$ we have
\begin{equation}
\left(\frac{n}{2}\right)^k \leq c_n(k)\leq n^k.
\label{eq: tcom5}
\end{equation}
\label{lemma: est}
\end{lemma}

\begin{proof} We argue by induction on $n$. For $n=3$ the inequalities follow from (\ref{eq: tcom6}). As for the inductive step, observe that   if $k<n+1$, then  (\ref{eq: tcom3}) implies that 
\[
c_{n+1}(k)\leq \ \sum_{j=0}^k \binom{k}{j} n^j= (n+1)^k
\]
and
\[
c_{n+1}(k) \geq   \sum_{j=0}^k \binom{k}{j}\left(\frac{n}{2}\right)^j =\left(1+\frac{n}{2}\right)^k>\left(\frac{n+1}{2}\right)^k.
\]
Next we deduce from (\ref{eq: tcom4}) and the induction assumption  that
\[
c_{n+1}(n+1)\leq \sum_{k=0}^n k\binom{n}{k} n^k\leq  (n+1)\sum_{k=0}^n \binom{n}{k} n^k= (n+1)^{n+1}
\]
\[
c_{n+1}(n+1)\geq \sum_{k=0}^n k\binom{n}{k}\left(\frac{n}{2}\right)^k.
\]
If we let
\[
B_n(t):=(1+t)^n,\;\; D_n(t):= tB_n'(t)= nt(1+t)^{n-1}
\]
then we deduce that
\[
\sum_{k=0}^n k\binom{n}{k}\left(\frac{n}{2}\right)^k= D_n\bigl(n/2\bigr)= \frac{n^2}{2}\left(1+\frac{n}{2}\right)^{n-1}.
\]
It remains to check that
\begin{equation}
\frac{n^2}{2}\left(1+\frac{n}{2}\right)^{n-1}\geq \left(\frac{n+1}{2}\right)^{n+1},\;\;\forall n\geq 3.
\label{eq: tcom9}
\end{equation}
Indeed, observe that (\ref{eq: tcom9}) is equivalent to  the inequality
\[
\left(\frac{n+2}{n+1}\right)^{n-1}\geq \frac{(n+1)^2}{2n^2},\;\;\forall n\geq 3,
\]
which is   holds  since the right-hand-side is $\leq 1$, $\forall n\geq 3$.
 \end{proof}
 
 Using  (\ref{eq: tcom7}) and (\ref{eq: tcom9}) we deduce that
 \[
 \frac{1}{4^{\frac{n(n-1)}{2}}} \sum_{k=0}^n \binom{n}{k}\left(\frac{n}{2}\right)^k \leq  \bsP(\bsK_n)\leq  \frac{1}{4^{\frac{n(n-1)}{2}}} \sum_{k=0}^n \binom{n}{k}n^k.
 \]
This proves  (\ref{eq: comp1}) and completes the proof of Theorem \ref{th: comp}.
 
 \end{proof}

 \section{Final comments}
 \label{s: 4}
 \setcounter{equation}{0}

We   want  to extract some qualitative information from the quantitative  results proved so far.       The      invariant $\bh(\Gamma)$  enjoys a monotonicity. More precisely
\begin{equation}
\bh(\Gamma)\subset \bh(\Gamma')\;\;\mbox{ if $V(\Gamma)\subset V(\Gamma')$ and $ E(\Gamma)\supset E(\Gamma')$}
\label{eq: monotone}
\end{equation} 
Indeed,  we have
\[
\bsP(\Gamma)\leq \bsP(\Gamma'), \;\; N_\Gamma\geq N_{\Gamma'}.
\]
Next we observe that
\begin{equation}
\bh(\Gamma\sqcup \Gamma')= \frac{N_\Gamma}{N_\Gamma+ N_{\Gamma'}}\bh(\Gamma)+ \frac{N_{\Gamma'}}{N_\Gamma+ N_{\Gamma'}}\bh(\Gamma').
\label{eq: conv}
\end{equation}
If we let $\Gamma$ be a complete  graph with a large number of vertices and $\Gamma'$ be  a disjoint union  of $m$ edges, then
\[
\bh(\Gamma)\approx \log(1/4),\;\;\bh(\Gamma')=\log(3/4)
\]
and by varying $m$  we obtain  from (\ref{eq: conv}) the following result.

\begin{corollary} The discrete set  $\{\bh(\Gamma)\}$ is dense in the interval  $[\log(1/4),\log(3/4)]$.  \qed
\label{cor:  dense}
\end{corollary}

 Clearly, if $\Gamma_0$ and $\Gamma_1$ are isomorphic graphs then $\bh(\Gamma_0)=\bh(\Gamma_1)$. Coupling this  with (\ref{eq: monotone}) we deduce that   for any $x\in(\,\log(1/4),\log(3/4)\,)$ the property
 \[
 \bh(\Gamma) \leq x
 \tag{$\bsP_x$}
 \label{tag: px}
 \]
 is a  monotone increasing graph property in the sense of \cite[\S 2.1]{Bo}.       We denote by $p_n(x, N)$ the probability conditional probability 
 \[
 \bsP_n(x, N)=\bsP\bigl(\,\bh(\Gamma)\leq x\;\bigl|\; |V(\Gamma)|=n,\;\;|E(\Gamma)|=N\,\bigr),
 \]
 where the  set of graphs  with $n$ vertices and $N$-edges is equipped  with the uniform probability measure. 
 
 The results of   \cite{BoTh, FK} show that the property  (\ref{tag: px}) admits a threshold.  This means that there exists   a function
 \[
 \bm_x: \bZ_{>0}\ra \bZ_{>0}
 \]
 such that 
 \[
 \lim_{n\to\infty} \bsP_n(x, N_n)=0\;\;\mbox{if  $\lim_{n\to\infty}\frac{N_n}{\bm_x(n)}=0$},
 \]
 and 
 \[
  \lim_{n\to\infty} \bsP_n(x, N_n)=1\;\;\mbox{if $\lim_{n\to\infty}\frac{N_n}{\bm_x(n)}=\infty$.}
  \]
  The above simple observations  raise some obvious questions.
  
 \begin{question}     What more can one say about the   threshold $\bm_x$?\qed
 \end{question}
 
 \begin{question}  For $p\in (0,1)$ and $n$ a positive integer   we denote by $\eG(n,p)$ the set of  graphs with $n$ vertices  in which  the  edges are included independently with probability  $p$.  In $\eG(n,p)$ a graph with $N$  edges has probability $p^Nq^{E_n-N}$, where
 \[
 q:=(1-p),\;\;E_n:=\binom{n}{2}.
 \]
 The correspondence $\Gamma\mapsto\bh(\Gamma)$ determines  a  random variable
 \[
 \bh_p: \eG(n,p)\to R:=\bigl[\,\log(1/4), \log(3/4)\,\bigr]\cup\{0\}.
 \]
 Given    a map $p:\bZ_{>0}\to (0,1)$, $n\mapsto p(n)$  what can be said about the  large $n$ behavior of the sequence of random variables $\bh_{p(n)}$ for various choices of $p(n)$'s? \qed
 
 \end{question}

 \begin{question} Denote  by $\eT_n$ the set of trees with vertex set $\{v_0,v_1,\dotsc, v_n\}$. For any $\Gamma\in\eT_n$ we have $N_\Gamma=n$, and any combinatorial flow on $\Gamma$ is  obviously acyclic. We set
 \[
 \bh_*(\eT_n)=\min_{\Gamma\in \eT_n}\bh(\Gamma),\;\;\bh^*(\eT_n)=\max_{\Gamma\in \eT_n}\bh(\Gamma),
 \]
 Observe that
 \[
 \bh_*(\eT_{n+1})\leq \bh_*(\eT_n),\;\;\bh^*(\eT_{n+1})\leq\bh^*(\eT_n).
 \]
 We set 
 \[
 \bh_*(\eT)=\lim_{n\to\infty}\bh(\eT_n),\;\;\bh^*:=\lim_{n\to\infty}\bh^*(\eT).
 \]
 Note  that
 \[
 \bh_*(\eT)\leq \log(1/2) <\log  r\leq  \bh^*(\eT),\;\;r=\frac{3+\sqrt{5}}{8}.
 \]
Is it true that
\[
\bh_*(\eT)=\log(1/2),\;\; \log  r= \bh^*(\eT)?
\]
\qed
\end{question}

\end{document}